\documentclass[12pt]{article}
\usepackage{mathrsfs}
\usepackage{dsfont}
\usepackage{amsthm}
\usepackage{mathrsfs}
\usepackage{amsmath}
\usepackage{amsfonts}
\usepackage[colorlinks,linkcolor=black,anchorcolor=blue,citecolor=blue]{hyperref}
\usepackage{amssymb, amsmath, cite}

\newtheorem{theorem}{Theorem}[section]
\newtheorem{lemma}{Lemma}[section]

\newcommand\be{\begin{equation}}
\newcommand\ee{\end{equation}}
\newcommand\ber{\begin{eqnarray}}
\newcommand\eer{\end{eqnarray}}
\newcommand\berr{\begin{eqnarray*}}
\newcommand\eerr{\end{eqnarray*}}

{\setlength\arraycolsep{2pt}

\begin{document}

\title{Existence of Vortices for Nonlinear Schr\"{o}dinger Equations}
\author{Shouxin Chen, Guange Su\footnote{E-mail address:
suguange@henu.edu.cn.}\\School of Mathematics and Statistics\\Henan University\\
Kaifeng, Henan 475004, PR China\\ }
\date{}
\maketitle

\begin{abstract}
In this paper, we study the existence of vortices for two kinds of nonlinear Schr\"{o}dinger equations arising from the Bose-Einstein condensates and geometric optics arguments, respectively. For the 
Gross-Pitaevskii equation from  Bose-Einstein condensates arguments, we introduce the weighted Sobolev space on which the corresponding functional is coercive. By using the variational methods, we prove the existence of positive and radially symmetric solutions under different types of boundary condition. And we study another equation arising from geometric optics arguments by constrained minimization method. 
Furthermore some explicit estimates for the bound of the wave propagation constant are also derived.

\end{abstract}

\medskip
\begin{enumerate}

\item[]



\end{enumerate}

\section{Introduction}
Vortices play an important role in many areas of modern physics such as condensed matter physics, particle interactions, cosmology, and quantum information processing. And the optical vortex \cite{lmr,ayl,ybl,gjc} is one of the branches. In the comprehensive work of Nye and Berry \cite{jm}, they described the optical vortex as dislocations or defects of waves. In brief, we may imagine the vortex as a ring located in the transverse plane, propagation along the vertical axis. In optics research, the complex-valued light waves, propagating in a nonlinear media, are governed by a nonlinear Schr\"{o}dinger equation \cite{ta,yvl,dcg,ama,dte}. Vortices also arises in the study of the Bose-Einstein condensates\cite{skl,zj,as}. Similarly, a fundamental prototype situation is when particle described by a complex-valued wave function governed by nonlinear Schr\"{o}dinger equations and it is referred to as the Gross-Pitaevskii equation \cite{dg}. Such theoretical studies provides a broad range of analytic problems related to the existence for mathematical investigation.

Petrov and Astrakharchik explored the weakly interacting Bose-Bose mixtures and showed that in the case of attractive inter- and repulsive intraspecies interactions the energy per particle has a minimum at a finite density corresponding to a liquid state\cite{dg}. They derived the Gross-Pitaevskii equation to describe droplets of such liquids. On the other hand, in \cite{asd} Dreischuh, Chervenkov and others considered the propagation of a beam in self-defocusing nonlinear medium with saturable nonlinearity whose evolution is also described by the normalized nonlinear Schr\"{o}dinger equation. Our aim in the present work is to obtain some existence theorems for the vortex solutions to these nonlinear Schr\"{o}dinger equations from the Bose-Einstein condensates and geometric optics arguments, respectively. The normalized nonlinear Schr\"{o}dinger equation has the following form
\begin{equation}\label{1}
\mathrm{i}\frac{\partial E}{\partial z}+\frac{1}{2}\triangle_{\bot}E-\psi_{i}(E)E=0,~i=1,2,
\end{equation}
where $E$~is a complex-valued field propagating in the~$z$-direction,~$\triangle_{\bot}$~is the Laplace operator over the plane of coordinates~$(x,y)$~which is perpendicular to the~$z$-direction, and functions $\psi_{i}(E)$ are
\begin{equation}\label{2}
 \psi_{1}(E)=\alpha|E|^{2}\ln(\frac{|E|^{2}}{\beta}),
\end{equation}
and
\begin{equation}\label{3}
 \psi_{2}(E)=\frac{|E|^{2}}{(1+s|E|^{2})^\gamma},
\end{equation}
respectively. $\psi_{1}$ is from the Bose-Einstein condensates research \cite{dg}, the positive constant $\alpha,\beta$ only depends on the medium. And $\psi_{2}$ arises from geometric optics arguments \cite{asd}, the parameters $s>0$ and $\gamma>2$ depend on the particular realization of the experiment, e.g. the properties of the nonlinear medium and the focusing conditions. We expect to find an n-vortex solution of (\ref{1}) of the form
\begin{equation}\label{4}
E(x_{1},x_{2},z)=E(r,\theta,z)=u(r)e^{\mathrm{i}(n\theta+ \omega z)},
\end{equation}
where~$r,\theta$~are polar coordinates over~$\mathbb{R}^{2}$,~$u(r)$~is the radial profile function which gives rise to the density of Bose Einstein condensate and the intensity of light waves, respectively, integer~$n\in\mathbb{Z}\setminus\{0\}$~is the winding number, and~$\omega\in\mathbb{R}$~is the wave propagation constant. Inserting~(\ref{4})~into~(\ref{1}), we arrive at the following~$n$-vortex equation
\begin{equation}\label{5}
u''(r)+\frac{1}{r}u'(r)-\frac{n^2}{r^2}u(r)-2\omega u(r)-2\psi_{i}(u)u(r)=0,~i=1,2
\end{equation}
\begin{equation}\label{6}
\psi_{1}(u)=\alpha|u|^{2}\ln(\frac{|u|^{2}}{\beta}),
\end{equation}
\begin{equation}\label{7}
 \psi_{2}(u)=\frac{|u|^{2}}{(1+s|u|^{2})^\gamma}.
\end{equation}

We emphasize the following boundary condition:
\begin{equation}\label{8}
u(0)=u(R)=0.
\end{equation}
The presence of the vortex core at~$r=0$~requires~ $u(0)=0$. It is natural for ring-shaped vortices that the particle's density and beam's
intensity decays to zero at infinity. In other words, the facts allow us to impose the boundary condition $u(R)=0$ for sufficiently large distance $R > 0$ away from the vortex core.
Note that we will take $R=\infty$ in problem (\ref{5})-(\ref{8}) with $\psi_{1}$.

Besides, it make sense to ask for solutions with non-homogeneous boundary conditions, namely,
\begin{equation}\label{9}
u(0)=0,\  \  u(\infty)=k,
\end{equation}
where~$k$~is the biggest zero of the equation~$\omega u+\psi_{1}(u)u=0$. In this case, the force vanishes only near the center.

The existence of vortices has been described in many papers\cite{vb,yr,c,Q,tm,m}. In particular, in \cite{yr}, Yang and Zhang gave two types of results on the existence of
optical vortices in a bounded domain in $\mathbb{R}^{2}$. Inspired by their results, Carlo Greco did a further research in \cite{c}, which involves
two cases of the optical vortices model: the self-focusing cubic nonlinearity, and the competing quintic and cubic nonlinearity. Their work motivates our mathematical analysis. However, the logarithmic term in our problem is difficult to allow the acquisition of a weak solution. Moreover, the corresponding functional is not well defined in $H^{1}((0,\infty);rdr)$. For this, we introduce the weighted Sobolev space as in \cite{c} on which the functional is coercive, so that we can consider to use the classical variational methods. On the other hand, We study equation (\ref{5}) with $\psi_{2}$ by a direct method, namely constrained minimization method. And the propagation constant $\omega$ arises as a
Lagrange multiplier due to the constraint. Then we have the following results.

\begin{theorem}
For any positive constant $\alpha$, $\beta$,  consider the two-point boundary value problem
with $\psi_{1}$.\\ 
{\rm(1)}\ Problem (\ref{5})-(\ref{8}) always has a positive solution, if there exists a positive constant $\mu>0$ such that
\begin{equation}{\label{e1}}
\mu\leq\omega<\frac{1}{2}\mathrm{e}^{-\frac{1}{2}}\alpha\beta.
\end{equation}
And the solutions decays exponentially at infinity.\\
{\rm(2)}\ Problem (\ref{5})-(\ref{9}) has at least one solution, if $0<\omega<\frac{3}{4}\mathrm{e}^{-1}\alpha\beta$.
\end{theorem}

\begin{theorem}
For some positive constant $s>0$ and $\gamma>2$, consider the two-point boundary value problem
with $\psi_{2}$.\\  
{\rm(1)}\ Problem (\ref{5})-(\ref{8}) always has a solution pair $(u, \omega)$, $u>0$, $\omega\in\mathbb{R}$.\\
{\rm(2)}\  Let $(u,\omega)$ be the solution pair of the problem (\ref{5})-(\ref{8}) obtained in part (1). Then
$\omega$ has a bound from below as well as from above, precisely
\begin{equation}
-\frac{(\gamma-1)^{(\gamma-1)}}{s\gamma^{\gamma}}-\sqrt{\frac{24(1+n^{2}(2\ln2-1))\pi}{s^{2}(\gamma-1)(\gamma-2)P_{0}}}
\leq\omega<0,\notag
\end{equation}
where $P_{0}>0$ will be given in section 4.
\end{theorem}

The organization of this paper as follows. In Section 2, for a class of Gross-Pitaevskii equation, we show the existence of vortices on infinite intervals by variational method. 
Moreover, we obtain that solutions decay exponentially at infinity. We then, in Section 3, prove the existence of vortices under non-homogeneous boundary conditions by variational method. Next, we concentrate on the nonlinear Schr\"{o}dinger equations arising from the geometric optical model. In section 4, we establish the existence theorem of vortices on bounded intervals through the constrained variational method. Furthermore, we derive some explicit estimates for the upper and lower bounds of wave propagation constants~$\omega$.

\section{Vortices for homogeneous problem}
In this section, via variational method, we will prove the existence of vortices of the following two-point boundary value problem,
\begin{equation}\label{a}
u''+\frac{1}{r}u'-\left(\frac{n^2}{r^2}+2\omega\right)u-2\alpha u^3\ln\left(\frac{u^2}{\beta}\right)=0,
\end{equation}
\begin{equation}\label{b}
u(0)=u(\infty)=0.\
\end{equation}


Suppose that (\ref{e1}) holds and set
\begin{equation}
a(r)=\frac{n^2}{r^2}+2\omega.\notag
\end{equation}
Then we could introduce a space $H$, which is the completion of the space $C_{0}^{\infty}$ equipped with the inner product
\begin{equation}
\langle u,v\rangle=\int_{0}^{+\infty}(u'(r)v'(r)+a(r)u(r)v(r))r\mathrm{d}r.\notag
\end{equation}
Clearly,~$u\in H$~implies~$u\in H_{loc}^{1}(0,+\infty)$, so~$u$~is continuous on~$(0,+\infty)$.

Moreover, $H$ has the following properties \cite{c}

{\rm(1)} If $u\in H$, then we have $\lim\limits_{r\rightarrow0}u(r)=0$, and $u=O(r^{-\frac{1}{2}})$ as $r\rightarrow\infty$;

{\rm(2)}\ For all $u\in H$, there exists a positive constant~$c_{1}$ such that $\|u\|_{\infty}\leq\sqrt{c_{1}}\|u\|$, namely,~$H$~is embedded in $L^{\infty}([0,+\infty))$;

{\rm(3)}\ $H\subset H^{1}((0,+\infty);r\mathrm{d}r)$,~$H$~is compactly embedded in~$L^{p}((0,+\infty);r\mathrm{d}r)$~for any~$p>2$.\\

In order to approach the problem consisting of~(\ref{a})~and~(\ref{b}), we write down the action functional
\begin{equation}\label{l0}
I(u)=\frac{1}{2}\int_{0}^{+\infty}\left\{(u'(r))^{2}+\left(\frac{n^2}{r^2}+2\omega\right)u^2(r)+\alpha u^4(r)\ln\frac{u^2(r)}{\beta}-\frac{\alpha}{2}u^4(r)\right\}r\mathrm{d}r.
\end{equation}
For convenience, we set
$$
\tilde{P}(s)=s^{4}\ln\frac{s^{2}}{\beta}.
$$
Obviously, $\lim\limits_{s\rightarrow0}{\tilde{P}(s)}/{s^2}=0$. It follows that there exists $\sigma>0$ such that $|\tilde{P}(s)|<s^{2}$, for $|s|<\sigma$. Since $u(r)\rightarrow0$ as $r\rightarrow0$ and $r\rightarrow\infty$, for any $\delta>0$ suitable small and $R>0$ sufficiently large,  we have
\begin{equation}
\int^{\delta}_{0}\tilde{P}(u)r\mathrm{d}r\leq\int^{\delta}_{0}|\tilde{P}(u)|r\mathrm{d}r\leq\int^{\delta}_{0}|u(r)|^{2}r\mathrm{d}r\leq c_{2}^{2}\|u\|^{2},
\end{equation}
and
\begin{equation}
\int^{\infty}_{R}\tilde{P}(u)r\mathrm{d}r\leq\int^{\infty}_{R}|\tilde{P}(u)|r\mathrm{d}r\leq\int^{\infty}_{R}|u(r)|^{2}r\mathrm{d}r\leq c_{2}^{2}\|u\|^{2},
\end{equation}
where $c_{2}$ is an embedding constant. Hence for any $u\in H$, $\tilde{P}(u)\in L^{1}((0,\infty);r\mathrm{d}r)$. Then, (\ref{l0}) is well defined on the space~$H$. And its critical points are weak solutions of~(\ref{a})~with boundary conditions~(\ref{b}). By standard bootstrap argument its are also smooth. In order to get critical points of~$I$, we shall apply variational method.

\begin{lemma}\label{l1}
 The action functional $I$ given in (\ref{l0}) is coercive on $H$.
\end{lemma}
\begin{proof}
Let $c_{2}>0$ be the constant for the embedding of $H$ in $L^{2}((0,+\infty); r\mathrm{d}r)$ such that $\int^{+\infty}_{0}u^{2}(r)r\mathrm{d}r\leq c_{2}^{2}\|u\|^{2}$. We denote
\begin{equation}
P(s)=s^{4}\ln\frac{s^{2}}{\beta}-\frac{s^{4}}{2}.\notag
\end{equation}
It is clear that $P(s)>0$ for any $|s|>\mathrm{e}^{\frac{1}{4}}\beta^{\frac{1}{2}}$. Since $\lim\limits_{s\rightarrow0}{P(s)}/{s^2}=0$, then there exists $\varepsilon_{0}>0$, $\delta>0$ small enough such that $|P(s)|\leq\varepsilon_{0} s^{2}$ for $|s|<\delta$. It follows that
\begin{equation}\label{l2}
\int_{\{r:|u(r)|<\delta\}}P(u)r\mathrm{d}r\geq-\int_{\{r:|u(r)|<\delta\}}|P(u)|r\mathrm{d}r\geq-\varepsilon_{0}\int_{\{r:|u(r)|<\delta\}}u^{2}(r)r\mathrm{d}r\geq-\varepsilon_{0} c^{2}_{2}\|u\|^{2}.
\end{equation}
Moreover,
\begin{equation}\label{l3}
\int_{\{r:\delta\leq|u(r)|\leq\mathrm{e}^{\frac{1}{4}}\beta^{\frac{1}{2}}\}}P(u)r\mathrm{d}r
\geq-\frac{\beta^{2}}{2}\int_{\{r:\delta\leq|u(r)|\leq\mathrm{e}^{\frac{1}{4}}\beta^{\frac{1}{2}}\}}r\mathrm{d}r=-\frac{\beta^{2}}{4}|\Omega|^{2},
\end{equation}
where we represent the measure of the set $\{r:\delta\leq|u(r)|\leq\mathrm{e}^{\frac{1}{4}}\beta^{\frac{1}{2}}\}$ in terms $|\Omega|$.
Inserting (\ref{l2}), (\ref{l3}) to (\ref{l0}), we have
\begin{equation}\label{l4}
\begin{aligned}
I(u)&=\frac{1}{2}\|u\|^{2}+\frac{\alpha}{2}\int^{\infty}_{0}P(u)r\mathrm{d}r\\
&=\frac{1}{2}\|u\|^{2}+\frac{\alpha}{2}\int_{\{r:|u(r)|<\delta\}}P(u)r\mathrm{d}r+\frac{\alpha}{2}\int_{\{r:\delta\leq|u(r)|\leq\mathrm{e}^{\frac{1}{4}}\beta^{\frac{1}{2}}\}}P(u)r\mathrm{d}r+\frac{\alpha}{2}\int_{\{r:|u(r)|\geq\mathrm{e}^{\frac{1}{4}}\beta^{\frac{1}{2}}\}}P(u)r\mathrm{d}r\\
&\geq\frac{1}{2}\|u\|^{2}-\frac{\alpha }{2}\int_{\{r:|u(r)|<\delta\}}|P(u)|r\mathrm{d}r-\frac{\alpha}{2}\int_{\{r:\delta\leq|u(r)|\leq\mathrm{e}^{\frac{1}{4}}\beta^{\frac{1}{2}}\}}|P(u)|r\mathrm{d}r\\
&\geq\frac{1-\varepsilon_{0}\alpha c^{2}_{2}}{2}\|u\|^{2}-c,
\end{aligned}
\end{equation}
where $c\geq0$ only depends on $\alpha$, $\beta$ and $|\Omega|$. Setting $\varepsilon_{0}<\frac{1}{\alpha c_{2}^{2}}$, then we have $1-\varepsilon_{0}\alpha c_{2}^{2}>0$, and the lemma is proved.
\end{proof}

The above preparation allow us to consider the following optimization problem
\begin{equation}\label{l5}
\min\{I(u)|u\in H\}.
\end{equation}

\begin{lemma}\label{l6}
The problem (\ref{l5}) has a positive solution if (\ref{e1}) holds.
\end{lemma}
\begin{proof}
Let $\{u_{m}\}$ be a minimizing sequence of (\ref{l0}). Since the functional (\ref{l0}) is even,  we have $I(u_{m})\geq I(|u_{m}|)$, where we have also used the basic fact \cite{dn} that
for any function $u$ its distributional derivative must satisfy $||u|'_{r}|\leq |u'_{r}|$. Thus we may assume that the sequence $\{u_{m}\}$ consists of nonnegative functions. From (\ref{l4}), we see immediately that $\{u_{m}\}$ is bounded sequence in $H$. Hence there exists $u\in H$ such that $u_{m}\rightarrow u$ weakly in $H^{1,2}((0,\infty), rdr)$ as $m\rightarrow\infty$. And using the Rellich- Kondrachov theorem, we obtain that $u_{m}\rightarrow u$ strongly in $C[\delta,R]$ as $m\rightarrow\infty$ for any $0<\delta<R<\infty$.

Using Vitali Convergence theorem \cite{pn},  we have
\begin{equation}
\lim_{m\rightarrow\infty}\int_{0}^{R}u^{4}_{m}\ln\frac{u^{2}_{m}}{\beta}r\mathrm{d}r=\int_{0}^{R}u^{4}\ln\frac{u^{2}}{\beta}r\mathrm{d}r.
\end{equation}
Besides,
\begin{equation}
\begin{aligned}
&\left|\int^{\infty}_{R}u^{4}_{m}\ln\frac{u^{2}_{m}}{\beta}r\mathrm{d}r-\int_{R}^{\infty}u^{4}_{m}\ln\frac{u^{2}_{m}}{\beta}r\mathrm{d}r\right|\\
\leq&\int^{\infty}_{R}u_{m}^{2}\left|u_{m}^{2}\ln\frac{u_{m}^{2}}{\beta}-u^{2}\ln\frac{u^{2}}{\beta}\right|r\mathrm{d}r+\int^{\infty}_{R}\left|u^{2}\ln\frac{u^{2}}{\beta}\right||u_{m}^{2}-u^{2}|r\mathrm{d}r\\
\leq&\sup_{r\in(R,\infty)}\left|u_{m}^{2}\ln\frac{u_{m}^{2}}{\beta}-u^{2}\ln\frac{u^{2}}{\beta}\right|\int^{\infty}_{R}u_{m}^{2}r\mathrm{d}r+\sup_{r\in(R,\infty)}\left|u^{2}\ln\frac{u^{2}}{\beta}\right|\int^{\infty}_{R}|u_{m}^{2}-u^{2}|r\mathrm{d}r\\
\end{aligned}
\end{equation}
which approaches zero uniformly fast as $R\rightarrow\infty$ by using $u=O(r^{-\frac{1}{2}})$ as $r\rightarrow\infty$. Thus, we get
\begin{equation}\label{l7}
\lim_{m\rightarrow\infty}\int_{0}^{\infty}u^{4}_{m}\ln\frac{u^{2}_{m}}{\beta}r\mathrm{d}r=\int_{0}^{\infty}u^{4}\ln\frac{u^{2}}{\beta}r\mathrm{d}r.
\end{equation}

Since $H^{1}((0,\infty);r\mathrm{d}r)\rightarrow L^{p}((0,\infty);r\mathrm{d}r)$ is compact for $p>2$, we conclude that
\begin{equation}\label{l8}
\lim_{m\rightarrow\infty}\int_{0}^{\infty}u^{4}_{m}r\mathrm{d}r=\int_{0}^{\infty}u^{4}r\mathrm{d}r.
\end{equation}
Hence the negative terms on the right-hand side of (\ref{l0}) could be controlled when we consider the limiting behavior of $I$ over the minimizing sequence $\{u_{m}\}$.
We are now ready to show that the limit function $u$ is a minimizer of the problem (\ref{l5}). To proceed, we rewrite the functional (\ref{l5}) evaluated over the minimizing
sequence $\{u_{m}\}$ as
\begin{equation}\label{l9}
I(u_{m})=\frac{1}{2}\int_{0}^{+\infty}\left\{(u'_{m}(r))^{2}+\left(\frac{n^2}{r^2}+2\omega\right)u^2_{m}(r)+\alpha u^4_{m}(r)\ln\frac{u^2_{m}(r)}{\beta}-\frac{\alpha}{2}u^4_{m}(r)\right\}r\mathrm{d}r.
\end{equation}
Taking $m\rightarrow\infty$ in (\ref{l9}) and applying (\ref{l7}) and (\ref{l8}), we immediately arrive at the desired conclusion
$$
I(u)\leq\liminf_{m\rightarrow\infty}I(u_{m}).
$$
Consequently, $u$ is a solution to (\ref{l5}), but we must prove that $u\neq0$. For this purpose we show that there exists $v\in H$ such that $I(v)<0$. Because of (\ref{e1}), the function $Q(s)=s^{2}(2\omega+\alpha s^{2}\ln\frac{s^{2}}{\beta}-\frac{\alpha}{2}s^{2})$ is negative for some $k>0$, namely, $Q(k)<0$.
Let us consider the function $v:[0,\infty]\rightarrow\mathbb{R}$ such that
\begin{equation}
v(r)=\left\{
\begin{array}{lll}
   kr, & & { 0\leq r<1,}\\
   k, & & { 1\leq r <R,}\\
   k\mathrm{e}^{\lambda(R-r)}, & & { r\geq R,}
    \end{array}\right.\notag
\end{equation}
where $\lambda$ is a positive constant.
We claim that there exists $R$ large enough such that $I(v)<0$. Since $v(r)\leq k$ and $v'(r)=0$ on $[1,R-1]$, a simple calculation shows that
\begin{equation}\label{e4}
\int^{R}_{1}(v'^{2}+\frac{n^{2}}{r^{2}}v^{2})r\mathrm{d}r=\frac{k^{2}}{2}(2R-1)+n^{2}k^{2}\ln R.
\end{equation}
Furthermore,
\begin{equation}\label{e5}
\int^{R-1}_{1}Q(v(r))r\mathrm{d}r=\frac{Q(k)}{2}(R^{2}-1),
\end{equation}
and
\begin{equation}\label{e6}
\int^{R}_{\infty}Q(v(r))r\mathrm{d}r=C_{1}R+C_{2},
\end{equation}
where $C_{1}$ and $C_{2}$ are only dependent on $\alpha$, $\beta$, $k$ and $\lambda$.
From (\ref{e4}), (\ref{e5}) and (\ref{e6}), we have
$$
I(v)=\frac{1}{2}Q(k)R^{2}+O(R)~~~\text{for~}R\rightarrow\infty.
$$
Then, the claim follows.

To prove the existence of a positive solution of the boundary value problem (\ref{a})-(\ref{b}), we suppose that there exist a point $r_{0}$ such that $u'(r_{0})=0$. Applying the uniqueness theorem for the initial value problem of ordinary differential equations, we have $u(r)=0$ for all $r\in(0,\infty)$, which contradicts the fact that $u$ is a non-trivial solution.
\end{proof}

Next we will prove the solution $u$ decays exponentially at infinity. For convenience, we rewrite (\ref{a}) with $P'(u)$ as
\begin{equation}{\label{p1}}
u''+\frac{1}{r}u'=\frac{n^{2}}{r^{2}}+2\omega u+\frac{\alpha}{2} P'(u).
\end{equation}
If $u\in C^{0}[0,\infty)\cap C^{2}(0,\infty)$, then multiplying (\ref{p1}) by $u'r^{2}$, integrating over $[r_{1},r_{2}]\subset(0,\infty)$, we have
\begin{equation}{\label{p2}}
\begin{aligned}
u'^{2}(r_{2})r_{2}^{2}-u'^{2}(r_{1})r_{1}^{2}=&n^{2}(u^{2}(r_{2})-u^{2}(r_{1}))+2\omega(u^{2}(r_{2})r_{2}^{2}-u^{2}(r_{1})r_{1}^{2})+\alpha(P(u(r_{2}))r_{2}^{2}-P(u(r_{1}))r_{1}^{2})\\
&-4\omega\int^{r_{2}}_{r_{1}}u^{2}r\mathrm{d}r-2\alpha\int^{r_{2}}_{r_{1}}P(u)r\mathrm{d}r.
\end{aligned}
\end{equation}

\begin{lemma}\label{p3}
Let $u$ be a non-trivial solution of (\ref{a})-(\ref{b}). Then
$$
\lim_{r\rightarrow0}u'(r)r=0.
$$
\end{lemma}
\begin{proof}
From (\ref{p2}), applying Cauchy convergence principle, we have the limit $\lim\limits_{r\rightarrow0}u'^{2}(r)r^{2}$ exists, saying $l$. If $l>0$, then there exist $\delta>0$ and $0<\sqrt{l_{1}}<l$, such that $|u'(r)|>\frac{\sqrt{l_{1}}}{r}$ for any $r\in(0,\delta)$. By integrating on $[0,r]\subset[0,\delta)$, we obtain a contradiction to the fact $u(0)=0$. Then $l=0$ follows.
\end{proof}

\begin{lemma}
Let u be a non-trivial solution of (\ref{a})-(\ref{b}). Then $u$ decays exponentially as $r\rightarrow\infty$.
\end{lemma}
\begin{proof}
Step 1. We claim that $u$ does not change sign on some interval $[R,\infty)$. Otherwise, there exists a sequence $\{r_{m}\}$, with $r_{m}\rightarrow\infty$, such that
$u(r_{m})>0$, $u'(r_{m})=0$, $u''(r_{m})\leq0$. Then, from (\ref{p1}) we have
$$
\begin{aligned}
u''(r_{m})&=u''(r_{m})+\frac{1}{r_{m}}u'(r_{m})=\left(\frac{n^{2}}{r^{2}_{m}}+2\omega+\frac{\alpha}{2}\frac{P'(u(r_{m}))}{u(r_{m})}\right)u(r_{m})\\
&\geq\left(2\mu+\frac{\alpha}{2}\frac{P'(u(r_{m}))}{u(r_{m})}\right)u(r_{m}),
\end{aligned}
$$
thus we get a contradiction since $P'(s)/s\rightarrow0$ as $r\rightarrow0$. Without loss of generality, we can assume that $u(r)>0$ on $[R,\infty)$, for $R$ big enough.

Step 2. We now show that $u'(r)<0$ on $[R_{1},\infty)\subset[R,\infty)$. Since
$$
(u'(r)r)'=u''(r)r+u'(r)\geq\left(2\mu+\frac{\alpha}{2}\frac{P'(u(r))}{u(r)}\right)u(r)r,
$$
there exists $R_{1}\geq R$ such that $(ru')'>0$ for any $r\in [R_{1},\infty)$, which implies $u'(r)r$ strictly increasing on $[R_{1},\infty)$. It follows that the limit $\lim\limits_{r\rightarrow\infty}u'(r)r$ exists. By a similar argument of Lemma \ref{p3}, we have $\lim\limits_{r\rightarrow\infty}u'(r)r=0$. Consequently, we get $u'(r)<0$.

Step 3. To prove the exponential decay, we see that on $[R_{1},\infty)$
$$
u''(r)-\left(2\omega+\frac{\alpha}{2}\frac{P'(u(r))}{u(r)}\right)u(r)=\frac{n^{2}}{r^{2}}u(r)-\frac{1}{r}u'(r)>0,
$$
so that
$$
u''(r)>\left(2\omega+\frac{\alpha}{2}\frac{P'(u(r))}{u(r)}\right)u(r).
$$
Since $P'(s)/s\rightarrow0$ as $r\rightarrow0$, we can find some $c>0$ such that $u''(r)>cu(r)$. Multiplying by $u'(r)<0$, and integrating on $[r,\infty]$, we have $u'^{2}(r)>cu^{2}(r)$, so that
$$
u'(r)/u(r)<-\sqrt{c}.
$$
By integrating again on $[r_{1},r]$, we obtain
$$
0<u(r)<u(r_{1})\mathrm{e}^{-\sqrt{c}(r-r_{1})}.
$$
\end{proof}


\section{Vortices for non-homogeneous problem}
In this section we follow the idea of \cite{c}, using a constructive argument to obtain the vortex solutions of the following two-point boundary value problem with $0<\omega<{3\alpha\beta}/{4\mathrm{e}}$.
\begin{equation}\label{l1}
u''+\frac{1}{r}u'-\left(\frac{n^2}{r^2}+2\omega\right)u-2\alpha u^3\ln\left(\frac{u^2}{\beta}\right)=0,
\end{equation}
\begin{equation}\label{12}
u(0)=0,~u(\infty)=k.
\end{equation}
It is easy to check $k>\sqrt{\frac{\beta}{\mathrm{e}}}$.

Define
\begin{equation}
g(t)=\left\{
 \begin{array}{lll}
   2\omega t+2\alpha t^3\ln(\frac{t^2}{\beta}), & & {t\geq0,}\\
    2\omega t, & & {t<0.}
    \end{array}
    \right.\notag
\end{equation}
Then we consider the equation
\begin{equation}\label{13}
u''+\frac{1}{r}u'=\frac{n^2}{r^2}u+g(u)
\end{equation}
with the boundary conditions as in (\ref{12}). We claim that the solution of (\ref{13})-(\ref{12}) is positive. Otherwise, since $u(r)\rightarrow k>0$ as $r\rightarrow\infty$, there exist $r_{1}$ and $r_{2}$, $0\leq r_{1}<r_{2}$ such that $u(r_{1})=u(r_{2})=0$ and $u(r)<0$ in $(r_{1},r_{2})$. Then applying the maximum principle to (\ref{13}) with $g(u)=2\omega u$ over $[r_{1},r_{2}]$, we immediately obtain a contradiction. Thus, a solution of (\ref{13})-(\ref{12}) solves two-point boundary value problem (\ref{l1})-(\ref{12}).

It is necessary to reduce the non-homogeneous problem (\ref{13})-(\ref{12}) to a homogeneous one. For this, let $\varphi:[0,\infty)\rightarrow[0,k]$ be a smooth function with $\varphi=0$ on $[0,1]$ and $\varphi=k$ on $[2,\infty)$. Set
\begin{equation}\label{j1}
\eta(r)=\varphi''(r)+\frac{1}{r}\varphi'(r)-\frac{n^2}{r^2}\varphi(r).
\end{equation}
Clearly, $\eta(r)$ is smooth and $\eta=0$ on $[0,1]$ and $\eta=-{n^{2}k}/{r^{2}}$ on $[2,\infty)$. Then $u(r)$ is a solution of (\ref{13})-(\ref{12}) if and only if the function $v(r)=u(r)-\varphi(r)$ is a solution of
\begin{equation}\label{14}
v''+\frac{1}{r}v'=\frac{n^2}{r^2}+g(v+\varphi)-\eta(r),
\end{equation}
with boundary conditions
\begin{equation}\label{15}
v(0)=v(\infty)=0.
\end{equation}

In order to introduce a suitable functional framework for problem (\ref{14})-(\ref{15}), we first show that $g'(k)>0$. Since $k>\sqrt{\beta/\mathrm{e}}$ and $g(k)=0$, we have $\omega+\alpha k^{2}\ln(k^{2}/ \beta)=0$, so that
\begin{equation}
g'(k)=2\omega+6\alpha k^2\ln\frac{k^2}{\beta}+4\alpha k^2=4\alpha k^2\left(\ln\frac{k^2}{\beta}+1\right)>0.\notag
\end{equation}
It make sense to consider the space $H_{1}$ defined as the closure of $C_{0}^{\infty}(0,+\infty)$ for the norm
\begin{equation}
\|v\|_{1}=\left(\int_{0}^{+\infty}(v'(r))^{2}r+
\left(\frac{n^2}{r^2}+g'(k)\right)v^2(r)r\mathrm{d}r\right)^\frac{1}{2}.
\end{equation}
The space $H_{1}$ is a slight variant of the space $H$ introduced in Section 2. So $H_{1}$ has the same properties with $H$. In particular, $v\in H_{1}$ implies that  $v$ is continuous and satisfies (\ref{15}). We denote $G$ a primitive of $g$, namely, $G'(s)=g(s)$. Let us consider the following functional:
\begin{equation}\label{j1}
I_{1}(v)=\int_{0}^{+\infty}\left\{\frac{1}{2}(v'(r))^{2}+\frac{1}{2}\frac{n^2}{r^{2}}v^2(r)+G(v(r)+\varphi(r))-G(k)-\eta(r)v(r)\right\}r\mathrm{d}r.
\end{equation}
Since $G(s)-G(k)=O((s-k)^{2})$ near $k$, and $\eta(r)=O(1/r^{2})$ for $r\rightarrow +\infty$, the functional $I_{1}$ is well defined on $H_{1}$, and its critical points are (smooth) solutions of (\ref{14})-(\ref{15}).

\begin{lemma}
 The action functional $I_{1}$ given in (\ref{j1}) is coercive on $H_{1}$, if $0<\omega<3\alpha\beta/4\mathrm{e}$.
\end{lemma}
\begin{proof}
It is easy to check that if $0<\omega<3\alpha\beta/4\mathrm{e}$, there exists $\bar{c}>0$ such that for all $s\in\mathbb{R}$, $G(s)-G(k)\geq\bar{c}(s-k)^{2}$. Then we can find a positive constant $a_{1}$ independent of $v$, such that
\begin{equation}
\begin{aligned}
\int_{0}^{+\infty}\left(G(v(r)+\varphi(r))-G(k)\right)r\mathrm{d}r&\geq\bar{c}\int_{0}^{+\infty}(v(r)+\varphi(r)-k)^2r\mathrm{d}r\\
&\geq\bar{c}\int_{0}^{+\infty}v^2(r)r\mathrm{d}r-a_{1}\left(\int_{0}^{+\infty}v^2(r)r\mathrm{d}r\right)^\frac{1}{2}.\notag
\end{aligned}
\end{equation}
And using H\"{o}lder inequality, we have
\begin{equation}
\int_{0}^{+\infty}\eta(r)v(r)r\mathrm{d}r\leq a_{2}\left(\int_{0}^{+\infty}v^2(r)r\mathrm{d}r\right)^\frac{1}{2},\notag
\end{equation}
for some $a_{2}>0$, so that
\begin{equation}\label{j2}
\int_{0}^{+\infty}\left(G(v(r)+\varphi(r))-G(k)\right)r\mathrm{d}r-\int_{0}^{+\infty}\eta(r)v(r)r\mathrm{d}r
\geq\frac{\bar{c}}{2}\int_{0}^{+\infty}v^2(r)r\mathrm{d}r-a_{3},
\end{equation}
where $c_{3}=(a_{1}+a_{2})^{2}/2\bar{c}$. Inserting (\ref{j2}) to (\ref{j1}), then we obtain the coercive inequality
\begin{equation}\label{16}
I_{1}(v)\geq\frac{1}{2}\int_{0}^{+\infty}\left((v'(r))^{2}+\left(\frac{n^2}{r^2}+\bar{c}\right)v^2(r)\right)r\mathrm{d}r-a_{3}\geq a_{4}\|v\|^2_{1}-a_{3},
\end{equation}
where $a_{4}=\min(1,\bar{c}/g'(k))/2$.
\end{proof}

We now use the Taylor series of $G(s)$ on $s=k$:
\begin{equation}
G(s)-G(k)=\frac{1}{2}G''(k)(s-k)^2+Q(s-k)=g'(k)(s-k)^2+Q(s-k),\notag
\end{equation}
where $Q(s)$ is a polynomial with degree$>2$, and rewrite the functional $I_{1}$ in the form
\begin{equation}
\begin{aligned}
I_{1}(v)=&\frac{1}{2}\|v\|^{2}_{1}+g'(k)\int_{0}^{+\infty}(\varphi(r)-k)v(r)r\mathrm{d}r
+\frac{1}{2}g'(k)\int_{0}^{+\infty}(\varphi(r)-k)^{2}r\mathrm{d}r\notag\\
&+\int_{0}^{+\infty}Q(v(r)+\varphi(r)-k)r\mathrm{d}r-\int_{0}^{+\infty}\eta(r)v(r)r\mathrm{d}r.
\end{aligned}
\end{equation}

\begin{lemma}
Let $\{v_{m}\}$ be a sequence on $H_{1}$ such that $|I_{1}(v_{m})| \leq c$ and $I'(v_{m})\rightarrow0$ as $n\rightarrow\infty$. Then a subsequence of ${v_{m}}$ converges strongly to some $v\in H_{1}$.
\end{lemma}
\begin{proof}
The coercive inequality (\ref{16}) gives us the bounded
\begin{equation}
\left(\int_{0}^{+\infty}(v_{m}'(r))^{2}r+\left(\frac{n^2}{r^2}+g'(k)\right)v^{2}_{m}(r)r\mathrm{d}r\right)^\frac{1}{2}\leq c,
\end{equation}
where $c>0$ is a constant independent of $m$. Thus, we can suppose that $v_{m}\rightarrow v$ in $H_{1}$. Consequently,
\begin{equation}
\langle v,v-v_{m}\rangle=o(1)~(m\rightarrow\infty).
\end{equation}

Since $I'_{1}(v_{m})\rightarrow0$ as $m\rightarrow$ and $(v-v_{m})$ is bounded, we have
\begin{equation}
\begin{aligned}
\langle I'_{1}(v_{m}),v-v_{m}\rangle
=&\langle v_{m},v-v_{m}\rangle+g'(k)\int_{0}^{+\infty}(\varphi(r)-k)(v-v_{n})r\mathrm{d}r\\
&+\int_{0}^{+\infty}Q'(v_{m}+\varphi(r)-k)(v-v_{m})r\mathrm{d}r-\int_{0}^{+\infty}\eta(r)(v-v_{m})r\mathrm{d}r\\
=&o(1)~(m\rightarrow\infty).\notag
\end{aligned}
\end{equation}
Clearly, we also have
\begin{equation}
g'(k)\int_{0}^{+\infty}(\varphi-k)(v-v_{m})r\mathrm{d}r-\int_{0}^{+\infty}\eta(r)(v-v_{m})r\mathrm{d}r=o(1)~(m\rightarrow\infty).\notag
\end{equation}
Applying the compact embedding $H_{1}((0,+\infty);r\mathrm{d}r)\rightarrow L^{3}((0,+\infty);r\mathrm{d}r)$, we see that
\begin{equation}
\begin{aligned}
\left|\int^{+\infty}_{0}(v_{m}+\varphi-k)^{p}(v-v_{m})r\mathrm{d}r\right|&\leq\left(\int_{0}^{+\infty}|v_{m}+\varphi-k|^{\frac{3p}{2}}r\mathrm{d}r\right)^{\frac{2}{3}}
\left(\int_{0}^{+\infty}|v-v_{m}|^{3}r\mathrm{d}r\right)^{\frac{1}{3}}\\
&=o(1)~(m\rightarrow\infty)\notag
\end{aligned}
\end{equation}
for any $p>2$. Hence
\begin{equation}
\int_{0}^{+\infty}Q'(v_{m}+\varphi(r)-k)(v-v_{m})r\mathrm{d}r=o(1)~(m\rightarrow\infty).\notag
\end{equation}
Then $\|v-v_{m}\|_{1}\rightarrow0~(m\rightarrow\infty$), the lemma is proved.
\end{proof}


\section{Vortices via constrained minimization}
In this section, we study the existence of optical vortices which are solutions of the boundary value problem
\begin{equation}\label{b00}
u''+\frac{1}{r}u'-\frac{n^{2}}{r^{2}}u-\frac{2u^{3}}{(1+su^{2})^{\gamma}}=2\omega u,
\end{equation}
\begin{equation}\label{b0}
u(0)=u(R)=0,
\end{equation}
as the critical point of the action functional
\begin{equation}\label{b1}
\begin{aligned}
J(u)=&\frac{1}{2}\int^{R}_{0}\left((u')^{2}+\frac{n^{2}}{r^{2}}u^{2}\right)r\mathrm{d}r
+\frac{1}{s^{2}(\gamma-1)(\gamma-2)}\int^{R}_{0}\left(1-\frac{1+\gamma su^{2}}{(1+su^{2})^{\gamma}}\right)r\mathrm{d}r\\
&-\frac{1}{\gamma-2}\int^{R}_{0}\frac{u^{4}}{(1+su^{2})^{\gamma}}r\mathrm{d}r,
\end{aligned}
\end{equation}
with the constraint functional
\begin{equation}\label{b2}
P(u)=\int|E|^{2}r\mathrm{d}r\mathrm{d}\theta=2\pi\int^{R}_{0}u^{2}r\mathrm{d}r.
\end{equation}
which measures the beam power of vortex wave. It suffices to prove the existence of a solution to the following constrained minimization problem
\begin{equation}\label{b3}
\min\left\{J(u)|u\in\mathcal{A}, P(u)=P_{0}>0\right\},
\end{equation}
where the admissible class $\mathcal{A}$ is defined by
\begin{equation}
\mathcal{A}=\left\{u(r) \textrm{ is absolutely continuous over } [0,R], u(R)=0, J(u)<\infty\right\},\notag
\end{equation}
$P_{0}$ is a prescribed value for the beam power, and $\omega$ arises as the Lagrange multiplier.

For convenience, we set
\begin{equation}
q(t)=\frac{1}{s^{2}(\gamma-1)(\gamma-2)}\left(1-\frac{1+\gamma st^{2}}{(1+st^{2})^{\gamma}}\right)
-\frac{t^{4}}{(\gamma-2)(1+st^{2})^{\gamma}},\notag
\end{equation}
where $\gamma>2$ and $s>0$ are parameters. It is easy to check that $q(t)$ attains its minimum at $t=0$, so that for all $t\in\mathbb{R}$, $q(t)\geq0$.

\begin{lemma}
The constrained minimization problem (\ref{b3}) has a positive solution.
\end{lemma}
\begin{proof}
Because $q(\tau)$ is non-negative, we have
\begin{equation}
\begin{aligned}
J(u)&=\frac{1}{2}\int^{R}_{0}\left((u')^{2}+\frac{n^{2}}{r^{2}}u^{2}\right)r\mathrm{d}r+\int^{R}_{0}q(u)r\mathrm{d}r\\
&\geq\frac{1}{2}\int^{R}_{0}\left((u')^{2}+\frac{n^{2}}{r^{2}}u^{2}\right)r\mathrm{d}r.\notag
\end{aligned}
\end{equation}
Let $\{u_{m}\}$ be a minimizing sequence of (\ref{b3}). Then we can find a positive constant $c$ independent of $m$ to get
\begin{equation}\label{b4}
\int^{R}_{0}(u'_{m})^{2}r\mathrm{d}r+\int^{R}_{0}\frac{1}{r}u^{2}_{m}\mathrm{d}r\leq c.
\end{equation}

Since both functionals $I$ and $P$ are even, and $||u|_{r}|\leq|u_{r}|$, we have $J(u_{m})\geq J(|u_{m}|)$, $P(u_{m})=P(|u_{m}|)$. This implies the sequence $\{u_{m}\}$ may be modified so that each $u_{m}\geq0$. Because of
\begin{equation}
\int^{R}_{0}u^{2}r\mathrm{d}r\leq R^{2}\int^{R}_{0}\frac{u^{2}}{r}\mathrm{d}r,\notag
\end{equation}
the minimizing sequence $\{u_{m}\}$ is bounded in $W^{1,2}((0,R),rdr)$. We may assume without loss of generality that for any $\varepsilon\in(0,R)$, $\{u_{m}\}$ converges weakly to an element $u\in W^{1,2}(\varepsilon,R)$ as $m\rightarrow\infty$. Using the compact embedding $W^{1,2}(\varepsilon,R)\rightarrow C[\varepsilon,R]$ , we have $u_{m}\rightarrow u$ in $C[\varepsilon,R]$  as $m\rightarrow\infty$. And $u=u(r)$, $u(R)=0$. Moreover, for any $0<r_{1}<r_{2}<R$, we have
\begin{equation}\label{b5}
\begin{aligned}
|u^{2}_{m}(r_{2})-u^{2}_{m}(r_{1})|&\leq2\left(\int^{r_{2}}_{r_{1}}(u'_{m})^{2}r\mathrm{d}r\right)^\frac{1}{2}\left(\int^{r_{2}}_{r_{1}}\frac{1}{r}u^{2}_{m}\mathrm{d}r\right)^\frac{1}{2}\\
&\leq2c^\frac{1}{2}\left(\int^{r_{2}}_{r_{1}}\frac{1}{r}u^{2}_{m}\mathrm{d}r\right)^\frac{1}{2},
\end{aligned}
\end{equation}
where the constant $c>0$ is as given in (\ref{b4}). Therefor, taking $m\rightarrow\infty$ in (\ref{b5}), we have
\begin{equation}\label{b05}
|u^{2}(r_{2})-u^{2}(r_{1})|\leq2c^\frac{1}{2}\left(\int^{r_{2}}_{r_{1}}\frac{1}{r}u^{2}\mathrm{d}r\right)^\frac{1}{2}.
\end{equation}

In view of Fatou's lemma, we have
\begin{equation}\label{b6}
\int^{R}_{0}(u')^{2}r\mathrm{d}r\leq\liminf_{m\rightarrow\infty}\int^{R}_{0}(u'_{m})^{2}r\mathrm{d}r,
\end{equation}
\begin{equation}\label{b7}
\int^{R}_{0}\frac{1}{r}u^{2}\mathrm{d}r\leq\liminf_{m\rightarrow\infty}\int^{R}_{0}\frac{1}{r}u^{2}_{m}\mathrm{d}r.
\end{equation}
Applying (\ref{b4}), (\ref{b05}) and (\ref{b7}), we immediately get $u(0)=0$.

The above consideration enables us to the conclusion that the limit of the minimizing sequence $\{u_{m}\}$ satisfies $u(0)=u(R)=0, u(r)\geq0$, and
\begin{equation}
J(u)\leq\liminf_{m\rightarrow\infty}J(u_{m}),~~P(u)=\lim_{m\rightarrow\infty}P(u_{m})=P_{0}.\notag
\end{equation}
Thus, $u$ is a solution to (\ref{b3}). Consequently, there is some $\omega\in\mathbb{R}$ such that $(\omega,u)$ solves (\ref{b00})-(\ref{b0}).

In the following, we will show that $u(r)$ is a positive solution. We may assume that there is a point $r_{0}\in(0,R)$ such that $u(r_{0})=0$. Since $r_{0}$ is a
minimum point for the function $u(r)$, we have $u'(r_{0})=0$. By a similar argument of Lemma 2.2, we obtain $u(r)=0$ for all $r\in(0,R)$, which contradicts the fact $P(u)=P_{0}>0$.
\end{proof}

\begin{lemma}
Let $(u,\omega)$ be a solution pair just obtained. Then
\begin{equation}
-\frac{(\gamma-1)^{(\gamma-1)}}{s\gamma^{\gamma}}-\sqrt{\frac{24(1+n^{2}(2\ln2-1))\pi}{s^{2}(\gamma-1)(\gamma-2)P_{0}}}
\leq\omega<0.
\end{equation}
\end{lemma}
\begin{proof}
We first claim
\begin{equation}\label{b8}
\liminf\limits_{r\rightarrow 0}\{r u|u'|\}=0.
\end{equation}
Suppose that (\ref{b8}) is false. Then there exist $c_{0}>0$ and ${r_{0}>0}$ such that
\begin{equation}
|r u(r)u'(r)|\geq c_{0}>0,~0<r<r_{0},\notag
\end{equation}
which implies that
\begin{equation}
\infty=\int^{r_{0}}_{0}\frac{c_{0}}{r}\mathrm{d}r\leq\int^{r_{0}}_{0}u|u'|\mathrm{d}r
\leq\left(\int^{r_{0}}_{0}\frac{1}{r}u^{2}\mathrm{d}r\right)^{\frac{1}{2}}
\left(\int^{r_{0}}_{0}(u')^{2}r\mathrm{d}r\right)^{\frac{1}{2}}.\notag
\end{equation}
This contradicts with $J(u)<\infty$. So we have $r_{j}u(r_{j})u'(r_{j})\rightarrow 0$ for some sequence $r_{j}\rightarrow0$ as $j\rightarrow\infty$.
Multiplying (\ref{b00}) by $u$, integrating over $[r_{j},R]$, and letting $j\rightarrow\infty$, we have
\begin{equation}\label{b9}
-\int^{R}_{0}(u')^{2}r\mathrm{d}r=\int^{R}_{0}\left(\frac{n^{2}}{r^{2}}+2\omega +\frac{2u^{2}}{(1+su^{2})^{\gamma}}\right)u^{2}r\mathrm{d}r.
\end{equation}
From (\ref{b9}) it is clear that if $\omega\geq0$, problem (\ref{b0})-(\ref{b00}) has only the trivial solution. Thus, to obtain non-trivial solution we need $\omega<0$. Therefore we obtain the upper estimate for $\omega$. Next, we consider lower bound of $\omega$.

Since $u$ solves (\ref{b3}), for any $u_{0}\in\mathcal{A}$, we have $J(u)\leq J(u_{0})$. Then,
\begin{equation}\label{b10}
\frac{1}{2}\int^{R}_{0}\left((u')^{2}+\frac{n^{2}}{r^{2}}u^{2}\right)r\mathrm{d}r\leq J(u_{0})-\int^{R}_{0}q(u)r\mathrm{d}r.
\end{equation}
Inserting (\ref{b10}) into (\ref{b9}), we get
\begin{equation}
\frac{1}{2\pi}\omega P_{0}\geq-J(u_{0})+\int^{R}_{0}q(u)r\mathrm{d}r-\int^{R}_{0}
\frac{u^{4}}{(1+su^{2})^\gamma}r\mathrm{d}r.\notag
\end{equation}
Because $q(t)$ is non-negative, for all $t\in\mathbb{R}$, it follows that
\begin{equation}
\begin{aligned}
\frac{1}{2\pi}\omega P_{0}&\geq-J(u_{0})-\int^{R}_{0}\frac{u^{4}}{(1+su^{2})^\gamma}r\mathrm{d}r\notag\\
&\geq-J(u_{0})-\max_{u\in(0,+\infty)}\frac{u^{2}}{(1+su^{2})^{\gamma}}\int^{R}_{0}u^{2}r\mathrm{d}r\notag\\
&=-J(u_{0})-\frac{(\gamma-1)^{(\gamma-1)}}{s\gamma^{\gamma}}\frac{P_{0}}{2\pi}.
\end{aligned}
\end{equation}

To estimate $J(u_{0})$, as in \cite{yr} we take $R=2a$ and define
\begin{equation}
u_{0}(r)=\left\{
\begin{array}{lll}
   \frac{b}{a}r, & & {0\leq r\leq a,}\\
    \frac{b}{a}(2a-r), & & {a < r\leq 2a.}
    \end{array}\right.\notag
\end{equation}
Then, we have
\begin{equation}\label{e11}
\begin{aligned}
J(u_{0})\leq b^{2}(1+n^{2}(2\ln2-1))+\frac{2a^{2}}{s^{2}(\gamma-1)(\gamma-2)}.
\end{aligned}
\end{equation}
By some simple calculation we get
\begin{equation}\label{e10}
P_{0}=2\pi\int_{0}^{2a}u_{0}^{2}(r)rdr=\frac{4\pi}{3}a^{2}b^{2},
\end{equation}
so that
\begin{equation}\label{e12}
a^{2}=\frac{3}{4\pi}\frac{P_{0}}{b^{2}}.
\end{equation}
Thus, using (\ref{e11}) and (\ref{e12}), we arrive at
\begin{equation}\label{e13}
J(u_{0})\leq b^{2}(1+n^{2}(2\ln2-1))+\frac{3}{2\pi b^{2}}\frac{P_{0}}{s^{2}(\gamma-1)(\gamma-2)}.
\end{equation}
Moreover, applying the Schwartz inequality, we have
\begin{equation}
b^{2}(1+n^{2}(2\ln2-1))+\frac{3}{2\pi b^{2}}\frac{P_{0}}{s^{2}(\gamma-1)(\gamma-2)}
\geq\sqrt{\frac{6(1+n^{2}(2\ln2-1))P_{0}}{\pi s^{2}(\gamma-1)(\gamma-2)}}.\notag
\end{equation}
Hence, if (\ref{e13}) is fulfilled for any $b\in(0,\infty)$, it requires
\begin{equation}\label{b12}
J(u_{0})\leq\sqrt{\frac{6(1+n^{2}(2\ln2-1))P_{0}}{\pi s^{2}(\gamma-1)(\gamma-2)}}.
\end{equation}
Inserting (\ref{b12}) into (\ref{b9}), we have
\begin{equation}\label{b13}
\omega\geq-\sqrt{\frac{24\pi(1+n^{2}(2\ln2-1))}{s^{2}(\gamma-1)(\gamma-2)P_{0}}}-\frac{(\gamma-1)^{(\gamma-1)}}{s\gamma^{\gamma}}.
\end{equation}
Then, the proof is completed.
\end{proof}


\end{document}